\documentclass[11pt]{amsart}

\usepackage{amssymb,amsmath,enumerate,epsfig}%,epsf}%t1enc
\usepackage{amsfonts,color}
\usepackage{a4,latexsym, multirow}
\usepackage{parskip}
\usepackage{hyperref}
\hypersetup{pdftitle=Segre4}

\newcommand{\C} {\mathbb{C}}
\newcommand{\Q} {\mathbb{Q}}

\newcommand{\F}{\mathbb{F}}
\newcommand{\Z}{\mathbb{Z}}

\newcommand{\PP}{\mathbb{P}}
\newcommand{\NS}{\mathop{\rm NS}}

\newcommand{\disc}{\mathop{\rm disc}}

\newcommand{\mX}{\mathcal X}

\newcommand{\mZ}{\mathcal Z}

\newcommand{\KK}{k}

\newtheorem{Theorem}{Theorem}[section]
\newtheorem{Proposition}[Theorem]{Proposition}
\newtheorem{Lemma}[Theorem]{Lemma}

\newtheorem{Corollary}[Theorem]{Corollary}
\newtheorem{Assumption}[Theorem]{Assumption}

\theoremstyle{remark}
\newtheorem{Remark}[Theorem]{Remark}

\newtheorem{Example}[Theorem]{Example}

\newtheorem{Convention}[Theorem]{Convention}
\newtheorem{Conclusion}[Theorem]{Conclusion}

\theoremstyle{definition}

\newtheorem{Definition}[Theorem]{Definition}

\begin{document}

\title{At most 64 lines on smooth quartic surfaces (characteristic 2)}

%\dedicatory{Dedicated to Wolf Barth}

\author{S\l awomir Rams}
\thanks{Funding   by ERC StG~279723 (SURFARI) is gratefully acknowledged (M.~Sch\"utt). Partially supported by National Science Centre, Poland,  grant 2014/15/B/ST1/02197 (S. Rams).}
\address{Institute of Mathematics, Jagiellonian University, 
ul. {\L}ojasiewicza 6,  30-348 Krak\'ow, Poland}
%\address{Current address: 
%Institut f\"ur Algebraische Geometrie, Leibniz Universit\"at
%  Hannover, Welfengarten 1, 30167 Hannover, Germany} 
\email{slawomir.rams@uj.edu.pl}

\author{Matthias Sch\"utt}
%\thanks{Funding   by ERC StG~279723 (SURFARI) is gratefully acknowledged (M.~Sch\"utt).}
\address{Institut f\"ur Algebraische Geometrie, Leibniz Universit\"at
  Hannover, Welfengarten 1, 30167 Hannover, Germany}
  \address{Riemann Center for Geometry and Physics, Leibniz Universit\"at
  Hannover, Appelstrasse 2, 30167 Hannover, Germany}
\email{schuett@math.uni-hannover.de}

%\subjclass[2010]{14J28; 14G10, 14J27, 14J50}
%%
\keywords{K3 surface, quartic surface, elliptic fibration, line}
%%
%%
% \thanks{Funding   by ERC StG~279723 (SURFARI) and  NCN grant N N201 608040 (S.~Rams)
%  is gratefully acknowledged.}

%
\date{March  15, 2017}

\begin{abstract}
Let $\KK$ be a field of characteristic $2$.
 We give a geometric proof
 that there are no smooth quartic surfaces $S \subset \PP^3_k$ 
 with more than 64 lines (predating work of Degtyarev which improves this bound to 60). 
 We also exhibit a smooth quartic containing 60 lines
 which thus attains the record in characteristic $2$.
% As a key step, we derive the sharp bound
%that any line meets at most 20 other lines on $S$. %even.
\end{abstract}
%
%
% \begin{abstract}
%This paper concerns K3 surfaces with automorphisms of order 11 in arbitrary characteristic.
%Specifically we study the wild case and prove that a general such surface in characteristic 11
%has Picard number 2.
%We also construct K3 surfaces with an automorphism of order 11 in every characteristic,
%and supersingular K3 surfaces whenever possible.
% \end{abstract}
%% 
 \maketitle

 \section{Introduction}
 \label{s:intro}
 
 This paper continues our study of the maximum number of lines
 on smooth quartic surfaces in $\PP^3$
 initiated in \cite{RS} and \cite{RS14}.
 Starting from Segre's original ideas and claims in \cite{Segre},
 we proved in \cite{RS}
 that a smooth quartic surface outside characteristics $2$ and $3$
 contains at most 64 lines,
 with the maximum attained by Schur's quartic \cite{schur}.
 In characteristic $3$, this specializes to the Fermat quartic
 which contains 112 lines, the maximum by \cite{RS14}.
 In characteristic $2$, however, both these quartics degenerate
 which opens the way to new phenomena.
In this paper we study these phenomena and give a  geometric proof
that the maximum number of lines still cannot exceed 64:

% 
% The aim of this paper is to 
%study the configurations of lines on certain quartic surfaces.
%In particular, we prove:
%
%\begin{Proposition}
%\label{prop1}
%\begin{enumerate}
%\item
%A line $\ell$ on a geometrically smooth quartic surface $S$ in $\PP^3_k$
%  intersects at most $20$ other lines provided that char$(k)\neq 2,3$.
%  \item
%  \label{b}
%If $\ell$ meets more than 18 lines on $S$,
%then $S$ can be given by a quartic polynomial
%$$
%x_3x_1^3+x_4x_2^3+x_1x_2 q(x_3,x_4)+x_3g(x_3,x_4)
%$$
%where $q, g\in \bar k[x_3,x_4]$ are homogeneous 
%of degree 2 resp.~3.
%\item
%\label{c}
%The line $\ell=\{x_3=x_4=0\}$ meets 20 lines on $S$ if and only if $x_4\mid g$.
%\end{enumerate}
%%  \in k[x_3, x_4]$ (resp. $g \in k[x_3, x_4]$) is  a homogeneous polynomial of degree $1$ (resp.~$4$).
%\end{Proposition}
%
%
%The above proposition will enable us to  prove the main theorem of this paper:
% 
\begin{Theorem}
\label{thm}
Let $\KK$ be a field of characteristic $p=2$.
Then any %geometrically 
smooth quartic surface over $k$ contains at most 64 lines.
\end{Theorem}

After this paper was written, A.~Degtyarev stated in \cite{Degtyarev},
partly based on machine-aided calculations from \cite{DIS},
that the bound of Theorem \ref{thm} can be improved to 60 in characteristic $2$.
The record is attained by a quartic with $S_5$-action which we shall exhibit explicitly in Section~\ref{s:ex}.
We point out that unlike in other characteristics (by work of us and Veniani \cite{RS}, \cite{RS14}, \cite{Veniani}),
there exist non-smooth quartic K3 surfaces with more lines than in the smooth case,
in fact with as many as 68 lines in characteristics $2$  (see Remark \ref{rem:68}),
indicating how special this situation is.

%that Theorem \ref{thm} is actually sharp.

We emphasize that originally we were expecting the bound  from characteristics $\neq 2,3$ to go up in characteristic $2$,
since just like in characteristic $3$,
there may be quasi-elliptic fibrations
and the flecnodal divisor may degenerate.
With this in mind,
our previous best bound ended up at $84$ in \cite[Prop.~1.3]{RS14}.
In contrast,
this paper will show that quasi-elliptic fibrations 
in characteristic $2$ cannot arise from lines on smooth quartics (see Proposition \ref{prop:qe}).
Then we will make particular use of special features of elliptic fibrations
in characteristic 2,
and of the Hessian, to  preserve the original bound of $64$.

%Surprising as the decrease in the maximum number of lines may be,
%it is in spirit comparable to the case of maximal singular fibers
%of elliptic K3 surfaces
%which have Kodaira types $I_{19}, I_{14}^*$ outside characteristic $2$,
%but decrease to $I_{18}, I_{13}^*$ in characteristic $2$ by \cite{S-max}.

The paper is organized as follows.
Section \ref{s:fibr} reviews genus one fibrations for smooth quartics with lines
with a special view towards quasi-elliptic fibrations.
In Section \ref{s:RH}, we discuss ramification types and the Hessian of a cubic in characteristic $2$
to derive Segre's upper bound for the 
the number of lines  met by  %valency  
a line of the so-called first kind on a smooth quartic surface.
Lines of the second kind are analyzed in Section \ref{s:2nd}, much in the spirit of \cite{RS}.
The proof of Theorem \ref{thm} is given in Sections \ref{s:triangle} - \ref{s:sqfree}
by distinguishing which  basic configurations of lines occur on the quartic.
The paper concludes with an example of a smooth quartic over $\F_4$ containing 60 lines over $\F_{16}$.

\begin{Convention} 
Throughout this note we work over an algebraically closed field $\KK$ of characteristic $p=2$,
since base change does not affect the validity of Theorem \ref{thm}.
\end{Convention}

 \section{Genus one fibration}
 \label{s:fibr}

 Let $S$ be a smooth quartic surface over an algebraically closed field $k$
 of characteristic $2$.
 Assuming that $S$ contains a line $\ell$,
 the linear system $|\mathcal O_S(1)-{\ell} |$ gives a pencil of cubic curves;
 explicitly these are obtained as residual cubics $C_t$
 when $S$ is intersected with the pencil of planes $H_t$ containing $\ell$.
 In particular, we obtain a fibration
\begin{equation} \label{eq:pi}
\pi: S\to\PP^1
\end{equation}
whose fibers are reduced curves of arithmetic genus one.
Note that in general there need not be a section,
and due to the special characteristic,
the general fiber need not be smooth,
i.e. the fibration may a priori be quasi-elliptic.
In fact, we shall instantly rule this latter special behaviour out,
but before doing so,
we note the limited  types of singular fibers (in Kodaira's notation \cite{K})
which may arise from a plane curve of degree $3$:

\begin{table}[ht!]
\begin{tabular}{c|l}
Kodaira type & residual cubic\\
\hline
$I_1$ & nodal cubic\\
$I_2$ & a line and a conic meeting transversally in 2 points\\
$I_3$ & 3 lines meeting transversally in 3 points\\
$II$ & cuspidal cubic\\
$III$ &a line and a conic meeting tangentially in a point\\
$IV$ & 3 lines meeting transversally in a point
\end{tabular}
\vspace{.3cm}
\label{tab:F}
\caption{Possible singular fibres of $\pi$}% \eqref{eq:pi}}
\end{table}

While this is already quite restrictive for any genus one fibration,
it determines the singular fibers of a quasi-elliptic fibration in characteristic $2$ completely:
the general fiber has Kodaira type $II$, 
and for Euler-Poincar\'e characteristic reasons,
there are exactly 20 reducible fibers, all of type $III$.
It turns out 
that this together with the theory of Mordell-Weil lattices
provides enough information to rule out quasi-elliptic fibrations in our characteristic $2$ set-up:

\begin{Proposition}
\label{prop:qe}
The fibration $\pi$ cannot be quasi-elliptic.
\end{Proposition}

\begin{proof}
Assume to the contrary that $\pi$ is quasi-elliptic.
Then $S$ automatically is unirational,
and thus supersingular,
i.e. the N\'eron-Severi group $\NS(S)$ has rank $22$
equalling the second Betti number;
endowed with the intersection pairing,
the N\'eron-Severi lattice has discriminant
\begin{eqnarray}
\label{eq:Artin}
\disc \NS(S) = -2^{2\sigma} \;\;\; \text{ for some } \; \sigma\in\{1,\hdots,10\}
\end{eqnarray}
by \cite{Artin}.
We will use the following basic result
whose proof  resembles that of \cite[Thm. 2]{ES-2}.

\begin{Lemma}
\label{lem:qe}
If $\pi$ is quasi-elliptic,
then it admits a section.
\end{Lemma}

\begin{proof}[Proof of Lemma \ref{lem:qe}]
If there were no section,
then $\pi$ would have multisection index $3$,
thanks to the trisection $\ell$.
Hence we can define an auxiliary integral lattice $L$ of the same rank
by dividing the general fiber $F$ by $3$:
\[
L = \langle \NS(S), F/3\rangle.
\] 
Since $L$ can be interpreted as index $3$ overlattice of $\NS(S)$,
we obtain
\[
\disc L = \disc \NS(S)/3^2.
\]
By \eqref{eq:Artin}, this is not an integer, despite $L$ being integral, giving the desired contradiction.
\end{proof}

We continue the proof of Proposition \ref{prop:qe} 
by picking a section of $\pi$ and denoting it by $O$.
Then $\ell$ induces a section $P$ of $\pi$
which is obtained fiberwise (or on the generic fiber)
from $P$ by Abel's theorem.
By the theory of Mordell-Weil lattices \cite{ShMW}
(which also applies to quasi-elliptic fibrations),
the class of $P$ in $\NS(S)$ is computed as follows.
Let $r$ be the number of reducible fibers
which are intersected by $O$ in the linear component,
and denote the respective component by $\ell_i$.
Let $\ell.O=s$.
We claim that
\[
P = \ell -2O +(4+2s) F - (\ell_1+\hdots+\ell_r).
\]
To see this, it suffices to verify the following properties,
using the fact that $P\equiv \ell$ modulo the trivial lattice generated by $O$ and fiber components:
\begin{itemize}
\item
$P$ meets every fiber with multiplicity one in a single component
(the linear component; this is assured by subtracting $2O$ and the $\ell_i$);
\item
$P^2=-2$ (giving the coefficient of $F$ in the representation of $P$).
\end{itemize}
But then the Mordell-Weil group of a jacobian quasi-elliptic fibration is always finite
(compare \cite[\S 4]{RuS}),
so $P$ has height zero.
Using $P.O=8+3s-r$ and the correction terms $1/2$ from each of the $20-r$ reducible fibers
where $O$ meets the conic while $P$ always meets the line,
we find
\[
h(P) = 2\chi(\mathcal O_S) + 2 P.O - \sum_v \mbox{corr}_v(P) = 10 + 6s-\frac 32 r.
\]
Since the equation $h(P)=0$ has no integer solution
(reduce modulo $3$!),
we arrive at the required contradiction.
\end{proof}

\begin{Remark}
\label{rem:68}
Once quasi-elliptic fibrations are excluded,
one can adopt the techniques from \cite{RS}, \cite{RS14}
to prove without too much difficulty
 that $S$ cannot contain more than 68 lines.
While this is still a few lines away from Theorem \ref{thm},
it is an interesting coincidence that there exists a one-dimensional family of non-smooth quartic K3 surfaces
(parametrized by $\lambda$),
i.e.~admitting only isolated ordinary double points as  singularities,
which contain as many as 68 lines:
\[
\mX = \{(x_1^3+x_2^3)x_3+\lambda x_2^3x_4+x_1x_2x_4^2+x_3^4=0\}.
\]
We found this family experimentally during our search for smooth quartics with many lines (see Section \ref{s:ex}).
Generically, there is only a single (isolated) singularity; it is located at $[0,0,0,1]$ and has type $A_3$.
The minimal resolution of a general member of the family is a supersingular K3 surface of Artin invariant $\sigma=2$.

Recently, Veniani proved in \cite{Veniani2} that 68 is indeed the maximum for the number of lines
on quartic surfaces with at worst isolated ordinary double point singularities in characteristic $2$,
and every surface attaining this maximum is projectively equivalent to a member of the above family.
\end{Remark}

\section{Ramification and Hessian}
\label{s:RH}

In this section,
we introduce two of the main tools for the proof of Theorem \ref{thm}.
It is instructive that both of them have different features in characteristic $2$ than usual.

\subsection{Ramification}
First we consider the ramification of the restriction of the morphism $\pi$
to the line $\ell$:
\[
\pi|_\ell: \ell \to \PP^1
\]
Since this morphism has degree $3$,
is always has exactly 1 or 2 ramification points in characteristic $2$
(because of Riemann-Hurwitz and wild ramification).
We distinguish the ramification type of $\ell$ according to the ramification points as follows:

\begin{table}[ht!]
\begin{tabular}{c|c}
ramification type & ramification points\\
\hline
$(1)$ & one simple\\
$(1,1)$ & two simple\\
$(1,2)$ & one simple, one double\\
$(2,2)$ & two double
\end{tabular}
\end{table}

The ramification type is relevant for our purposes
because often one studies the base change of $S$ over $k(\ell)$
where by definition the fibration corresponding to $\pi$ attains a section.
In fact, we will usually extend the base field to the Galois closure of $k(\ell)/k(\PP^1)$
where $\ell$ splits into three sections.
Note that the field extension $k(\ell)/k(\PP^1)$ itself is Galois
if and only if $\ell$ has ramification type $(2,2)$.
Encoded in the ramification,
one finds how the singular fibers behave under the base change,
and more importantly, how they are intersected by the sections
obtained from $\ell$.

\subsection{Hessian}

We now introduce the Hessian of the residual cubics $C_\lambda$.
To this end, we apply a projective transformation,
so that
\[
\ell=\{x_3=x_4=0\}\subset\PP^3.
\]
Then the pencil of hyperplanes in $\PP^3$ containing $\ell$ is given by
\[
H_\lambda:\;\; x_4 = \lambda x_3
\]
(including the hyperplane $\{x_3=0\}$ at $\lambda=\infty$,
so everything in what follows can be understood in homogenous coordinates of $\PP^1$
parametrising $H_\lambda$;
we decided to opt for the affine notation for simplicity).
The residual cubics $C_\lambda$ of $S\cap H_\lambda$ 
are given by a homogeneous cubic polynomial
\[
g \in k[\lambda][x_1,x_2,x_3]_{(3)}
\]
which is obtained from the homogeneous quartic polynomial $f\in k[x_1,x_2,x_3,x_4]$ 
defining  $S$
by substituting $H_\lambda$ for $x_4$ and factoring out $x_3$.
Outside characteristic $2$,
the points of inflection of $C_\lambda$
(which are often used to define a group structure on $C_\lambda$,
at least when one of them is rational)
are given by the Hessian
\[
\det \begin{pmatrix}
\frac{\partial ^2 g}{\partial x_i \partial x_j}
\end{pmatrix}_{1\leq i,j\leq 3}
\]
In characteristic $2$, however,
some extra divisibilities in the coefficients force us to modify the Hessian formally
using the $x_1x_2x_3$-coefficient $\alpha$ of $g$
until it takes the following shape 
(understood algebraically over $\Z$ in terms of the generic coefficients of $g$
before reducing modulo $2$ and substituting):
\[
h = \frac 14 \left(\frac 12 \det \begin{pmatrix}
\frac{\partial ^2 g}{\partial x_i \partial x_j}
\end{pmatrix}_{1\leq i,j\leq 3} - \alpha^2 g\right)
\in k[\lambda][x_1,x_2,x_3]_{(3)}.
\]
(These manipulations must be known to the experts
as they amount to the saturation of the ideal generated
by $g$ and its Hessian over $\Z[\lambda][x_1,x_2,x_3]$.)
In order to use the Hessian for considerations of lines on $S$,
Segre's key insight from \cite{Segre}
 was that $h$ vanishes on each linear component of a given residual cubic $C_{\lambda_0}$
 (or, if $C_{\lambda_0}$ is singular, but irreducible, in its singularity).
That is, any line in a fiber of $\pi$ (i.e. intersecting $\ell$)
gives a zero of the following polynomial $R$,
obtained by intersecting $g$ and $h$ with $\ell$ (i.e. substituting $x_3=0$)
and taking the resultant with respect to either remaining homogeneous variable:
\[
R =\mbox{resultant}_{x_1} (g(x_1,1,0), h(x_1,1,0)) \in k[\lambda].
\]
More precisely, one computes that $R$ has generically degree $18$,
and that each line contributes to the zeroes of $R$ on its own, i.e.
if some fiber contains three lines, then $R$ attains a triple root:

\begin{Lemma}
\label{lem:R}
In the above set-up,
assume that $\pi$ has
\begin{enumerate}
\item
a fiber of type $I_3$ or $IV$ at $\lambda_0$,
then $(\lambda-\lambda_0)^3\mid R$;
\item
a fiber of type $I_2$ or $III$ with double ramification at $\lambda_0$,
then $(\lambda-\lambda_0)^2\mid R$.
\end{enumerate}
\end{Lemma}

For degree reasons, one directly obtains the following upper bound for the \emph{valency} $v(\ell)$ of 
the line $\ell$,
i.e. the number of other lines on $S$ met by $\ell$:

\begin{Corollary}[Segre]
\label{cor:v18}
If $R\not\equiv 0$, then $v(\ell)\leq 18$.
\end{Corollary}

This makes clear
that we have to carefully distinguish whether $R$ vanishes identically or not.
Recall from \cite{Segre}, \cite{RS}
how this leads to the following terminology:

\begin{Definition}
The line $\ell$ is said to be of the \emph{second kind} if $R\equiv 0$.
Otherwise, we call $\ell$ a line of the \emph{first kind}.
\end{Definition}

We will next show that lines of the second kind behave essentially
as in characteristic $\neq 2,3$.
For lines of the first kind,
the different quality of ramification changes the situation substantially,
but it is not clear (to us) whether the valency bound from Corollary \ref{cor:v18} is still sharp.
(The lines on the record surface from Section \ref{s:ex} have all valency 17.)

%before proving subsequently that lines of the first kind
%show a crucially different behaviour;
%in particular, contrary to characteristics $\neq 2,3$,
%Corollary \ref{cor:v18} will not be sharp in characteristic $2$!

\section{Lines of the second kind}
\label{s:2nd}

Since lines of the second kind in characteristic $2$
turn out to behave mostly like in characteristics $\neq 2,3$,
we will be somewhat sketchy in this section.
That is, while trying to keep the exposition as self-contained as possible,
we will refer the reader back to \cite{RS} for the details whenever possible.

Let $\ell$ be a line of the second kind.
Then, by definition,
$\ell$ is contained in the closure of the flex locus on the smooth fibers.
This severely limits the way how $\ell$ may intersect the singular fibers.
As in characteristics $\neq 2,3$ in \cite{RS},
one obtains the following configurations
depending on the ramification:

\begin{Lemma}
\label{lem:ram}
A line of the second kind may intersect the singular fibers of $\pi$
depending on the ramification as follows:
\begin{table}[ht!]
\begin{tabular}{c|c|l}
ramification & fibre type & configuration\\
\hline
\hline
unramified & $I_1$ & 3 smooth points\\
&
$I_3$ & 1 smooth point on each component\\
&
$IV$ & 1 smooth point on each component\\
\hline
simple &
$II$ & 1 smooth point and  the cusp\\
\hline
%$III$ & 1 smooth point of the line component, the tacnode\\
double &
$I_1$ & tangent to the node\\
&
$I_2$ &  tangent to one of the nodes\\
&
$IV$ & the triple point
\end{tabular}
%\vspace{.3cm}
%\label{tab:F3}
%\caption{Support of the closure of the flex locus on the singular fibres}
\end{table}
\end{Lemma}

We emphasize that fibers of type $II$ and $III$ necessarily come with wild
ramification
in characteristic $2$;
in fact they impose on the discriminant a zero of multiplicity at least 4 by \cite{SS}.
(In mixed characteristic, i.e.~when specializing from characteristic zero to characteristic $2$,
this can often be explained as some fiber of type $II$ 
absorbing other irreducible singular fibers without changing its type, 
compare Lemma \ref{lem:ramconfig} with the results valid in characteristic zero from \cite[Lem.~4.3]{RS}.)

As in \cite{RS}  we continue to argue with the base change of $S$
to the Galois closure of $k(\ell)/k(\PP^1)$.
By construction, this sees $\ell$ split into three sections;
taking one as zero for the group law,
the others necessarily become $3$-torsion.
In practice, this implies that fibers of type  $I_1$ and $I_3$ have to even out
-- including two possible degenerations:
\begin{itemize}
\item
a pair of fibers $I_1, I_3$ might merge to type $IV$;
\item
two fibers of type $I_1$ might merge to $I_2$, still paired with two $I_3$ fibers.
\end{itemize}
In the next table, we thus group the fibers according to their Kodaira type in single entries $II, IV$,
pairs $(I_1, I_3)$ and triples $(I_2, I_3, I_3)$.
Since  fiber type $I_2$ automatically comes with double ramification
by Lemma \ref{lem:ram},
we can bound the possible configurations of singular fibers
(where the precise numbers depend on the index of wild ramification of the $II$ fibers):

\begin{Lemma}
\label{lem:ramconfig}
For a line of the second kind,
the singular fibers of $\pi$ can be configured as follows:
\begin{table}[ht!]
\begin{tabular}{c|c|l}
ramification types & fibre configuration & possible degenerations\\
\hline
\hline
$(1), (1,2)$ & $1 \times II$ &\\
& $\leq 5 \times (I_1, I_3)$ & $IV$, at most one triple $(I_2, I_3, I_3)$\\
\hline
$(1,1)$ & $2 \times II$ & \\
&
$\leq 4 \times (I_1, I_3)$ & $IV$\\
%\hline
%$(1,2)$ &
%1 $\times II$ & \\
%& $\leq 5 \times (I_1, I_3)$ & $IV$, at most one triple $(I_2, I_3, I_3)$\\
\hline
%$III$ & 1 smooth point of the line component, the tacnode\\
$(2,2)$ &
$6 \times (I_1, I_3)$ & $IV$, at most two triples $(I_2, I_3, I_3)$\\
\end{tabular}
%\vspace{.3cm}
%\label{tab:F3}
%\caption{Support of the closure of the flex locus on the singular fibres}
\end{table}
\end{Lemma}

\begin{Corollary}
\label{cor:v20}
Unless the line $\ell$ of the second kind has ramification type $(2,2)$,
it has valency 
\[
v(\ell)\leq 16.
\]
Otherwise one has $v(\ell)\leq 20$.
\end{Corollary}

It is due to this result that we will have to pay particular attention
to lines of the second kind with ramification type $(2,2)$.
As it happens, quartics containing such a line are not hard to parametrise;
in fact, a comparison with the proof of  \cite[Lemma 4.5]{RS} shows that
exactly the same argument as in characteristics $\neq 2,3$ goes through:

\begin{Lemma}[Family $\mZ$]
\label{lem:Z}
Let $\ell$ be a line of the second kind on a smooth quartic $S$ with ramification type $(2,2)$.
Then $S$ is projectively equivalent to a quartic in the family $\mZ$ given by the homogeneous polynomials
$$
x_3x_1^3+x_4x_2^3+x_1x_2q_2(x_3,x_4)+q_4(x_3,x_4) \in k[x_1,x_2,x_3,x_4]_{(4)}
$$
where $\ell=\{x_3=x_4=0\}$,
$q_2  \in k[x_3, x_4]_{(2)}$ and $q_4 \in k[x_3, x_4]_{(4)}$.
% is  a homogeneous polynomial of degree $2$ (resp.~$4$).
\end{Lemma}

We will not need the precise location of all singular fibers of $\pi$
in what follows,
but we would like to highlight the ramified singular fibers of Kodaira type $I_1$
at $\lambda = 0, \infty$.
These degenerate to type $I_2$ if and only if $x_3$ resp.~$x_4$ divides $q_4$
(unless the surface attains singularities, for instance if $q_4$ has  a square factor).
Note that if $S$ is smooth and taken as in Lemma \ref{lem:Z}, then
\[
v(\ell) = 18 \;\; \Longleftrightarrow \;\; x_3x_4\nmid q_4.
\]

For the record, we also note the following easy consequence of our considerations
which we will use occasionally to specialise to jacobian elliptic fibrations:

\begin{Corollary}
\label{cor:21}
If $\pi$ admits no section,
or if no two lines on $S$ intersect,
then $S$ contains at most 21 lines.
\end{Corollary}

\begin{proof}
In the first case, since any line disjoint from $\ell$ would give a section of $\pi$,
we know that all lines on $S$ meet $\ell$.
Hence the corollary follows from the combination of
Corollaries \ref{cor:v18} and \ref{cor:v20}.

In the second case, the lines give orthogonal $(-2)$ classes in $\NS(X)$.
Since the latter has signature $(1,\rho(X)-1)$ with $\rho(X)\leq b_2(X)=22$,
the claim follows.
\end{proof}

\section{Proof of Theorem \ref{thm} in the triangle case}
\label{s:triangle}

We will break the proof of Theorem \ref{thm} into three cases,
depending on which configurations of lines the smooth quartic $S$ admits.
They will be treated separately in this and the next two sections.
Throughout this section, we work with a smooth quartic $S$ satisfying the following assumption:

\begin{Assumption}
$S$ contains a triangle (or a star) composed of 3 lines.
\end{Assumption}

Equivalently (since $S$ is assumed to be smooth), there is a hyperplane $H$ 
containing the three said lines and thus splitting completely into lines on $S$:
\[
H\cap S = \ell_1+\hdots+\ell_4.
\]

\subsection{}

If neither of the lines $\ell_1,\hdots,\ell_4$ is of the second kind,
then each of them meets at most $18$ lines on $S$ by Corollary \ref{cor:v18}.
Since any line on $S$ meets $H$,
we find that $S$ contains at most 64 lines as claimed.

\subsection{}

If the lines are allowed to be of the second kind,
but not of ramification type $(2,2)$,
then Corollary \ref{cor:v20} implies again
that $S$ contains at most 64 lines. % as claimed.

\subsection{}

To complete the proof of Theorem \ref{thm} in the triangle case,
it thus suffices to consider the case
where one of the lines, say $\ell_1$,
is of the second kind with ramification type $(2,2)$.
Hence $X$ can be given as in Lemma \ref{lem:Z};
in particular, it admits a symplectic automorphism $\varphi$ of order 3
acting by
\[
\varphi[x_1,x_2,x_3,x_4] \mapsto [\omega x_1,\omega^2 x_2,x_3,x_4]
\]
for some primitive root of unity $\omega$.
Note that $\varphi$ permutes the lines $\ell_2,\ell_3,\ell_4$
(or of any triangle coplanar with $\ell_1$).
In particular, these three lines are of the same type.
As before, we continue to distinguish three cases:

\subsubsection{}
\label{ss:631}

If $\ell_2$ is of the first kind,
then consider the degree 18 polynomial $R\in k[\lambda]$
associated to the flex points of the genus one fibration induced by $\ell_2$.
Locate the singular fiber $\ell_1+\ell_3+\ell_4$ (of type $I_3$ or $IV$) at $\lambda=0$.
Then an explicit computation shows that
\[
\lambda^4\mid R.
\]
Since this divisibility exceeds the lower bound from Lemma \ref{lem:R},
we infer from the arguments laid out in Section \ref{s:RH}
that $\ell_2$ meets at most 14 lines outside the fiber at $\lambda=0$.
In total, this gives
\[
v(\ell_i)\leq 17, \;\;\; i=2,3,4.
\]
Together with Corollary \ref{cor:v20}, this shows that $S$ contains at most 63 lines.

\subsubsection{}

Similarly, if $\ell_2$ is of the second kind, but not of ramification type $(2,2)$,
then by Corollary \ref{cor:v20}, there are no more than 60 lines on $S$.

\subsubsection{}

We conclude the proof of Theorem \ref{thm} in the triangle case
by ruling out that $\ell_2$ is of the second kind with ramification type $(2,2)$.
(Over $\C$ this case leads either to Schur's quartic (containing 64 lines) or
to $S$ being singular (Lemma 6.2 in \cite{RS})),
but presently the situation differs substantially since there can only be two ramification points anyway.
Rescaling $x_3, x_4$, we can assume that the given line lies in the fiber at $\lambda=1$; 
this determines the $x_4^4$-coefficient of $q_4$, say.
Then, up to the action of $\varphi$,
it is given by
\[
\ell_2: \;\;
x_3+x_4=x_1+x_2+q_2(1,1)x_3=0
\]
One checks that $\ell_2$ has generically ramification type $(1,1)$;
this degenerates to type $(2,2)$ with $S$ continuing to be smooth
if and only if
\[
q_2'(1,1)q_2(1,1)^2+q_4'(1,1)=0
\]
where the prime indicates the formal derivative with respect to either $x_3$ or $x_4$.
Substituting for a coefficient of $q_4$,
this directly implies
\[
\lambda^6 \mid R
\]
where we have located the $I_3$ or $IV$ fiber $\ell_1+\ell_3+\ell_4$ at $\lambda=0$
as in \ref{ss:631}.
Solving for $R\equiv 0$, we inspect the first few coefficients of $R$.
Starting with the coefficient of $\lambda^7$,
a simple case-by-case analysis yields that $R$ can only vanish identically if $S$ is singular.
This completes the proof of Theorem \ref{thm} in the triangle case.
\qed

\section{Proof of Theorem \ref{thm} in the square case} \label{s:square}

Throughout this section, we work with a smooth quartic $S$ satisfying the following assumption:

\begin{Assumption}
\label{ass6}
$S$ contains neither a triangle nor a star composed of 3 lines,
but it does contain a square comprising 4 lines.
\end{Assumption}

We shall refer to this situation as the square case.
Our arguments are inspired by an approach due to Degtyarev and Veniani (see \cite{Veniani}).

\begin{Lemma}
\label{lem:v12}
If $S$ contains no triangles or stars,
then each line $\ell$ on $S$ has valency 
\[
v(\ell)\leq 12.
\]
\end{Lemma} 

\begin{proof}
Since the genus one fibration $\pi$ induced by $\ell$ cannot be quasi-elliptic by Proposition \ref{prop:qe},
the proof of Lemma \ref{lem:v12} amounts to a simple Euler-Poincar\'e characteristic computation
as the contributions of the singular fibers (including wild ramification) have to sum up to $24$.
Presently, since $S$ admits  no triangles and stars by assumption,
$\pi$ can only have singular fibers of Kodaira types $I_1, I_2, II, III$.
Hence there can be at most 12 fibers containing a line.
\end{proof}

Denote any 4 lines forming a square on $S$ by $\ell_1,\hdots,\ell_4$.
Order the lines such that $\ell_i.\ell_j=1$ if and only if $i\not\equiv j\mod 2$.
Consider the two residual (irreducible) conics $Q_{12}, Q_{34}$ such that the hyperplane class $H$ decomposes on $S$ as
\[
H = \ell_1+\ell_2+Q_{12} = \ell_3+\ell_4+Q_{34}.
\]
Then the linear system $|2H-Q_{12}-Q_{34}|$ induces a genus one fibration
\[
\psi: S\to \PP^1
\]
with fibers of degree $4$ -- one of them is exactly the square $D=\ell_1+\hdots+\ell_4$ of Kodaira type $I_4$.
Any line on $S$ is either orthogonal to $D$ and thus a fiber component of $\psi$,
or it gives a section or bisection for $\psi$,
thus contributing to the valency of one or two of the lines $\ell_i$.
In total, this gives the upper bound
\begin{eqnarray*}
\#\{\text{lines on } S\} & \leq & \#\{\text{lines in fibers of } \psi\} + \sum_{i=1}^4 (v(\ell_i)-2)\\
& \leq & \#\{\text{lines in fibers of } \psi\} + 40
\end{eqnarray*}
where the second equality follows from Lemma \ref{lem:v12}.
We shall now study the possible fiber configurations 
to derive the following upper bound for the number of lines on $S$ which will prove Theorem~\ref{thm} in the square case.

\begin{Proposition}
\label{prop:60}
Under Assumption~\ref{ass6},
the smooth quartic $S$ contains at most 60 lines.
\end{Proposition}

Before starting the proof of Proposition \ref{prop:60} properly,
we note the possible reducible fibers of $\psi$:

\begin{table}[ht!]
\begin{tabular}{c|l}
Kodaira type & configuration of curves\\
\hline
%$I_1$ & nodal cubic\\
$I_2/III$ & line + cubic or two quadrics\\
$I_3/IV$ & 2 lines and a quadric\\
%$II$ & cuspidal cubic\\
%$III$ &a line and a conic meeting tangentially in a point\\
$I_4$ & 4 lines forming a square
\end{tabular}
\vspace{.3cm}
\label{tab:F}
\caption{Possible reducible fibres of $\psi$}% \eqref{eq:pi}}
\end{table}

Using the fact that additive fibers necessarily come with wild ramification
(so that they contribute at least 4 to the Euler-Poincar\'e characteristic, see \cite[Prop.~5.1]{SS}),
one can easily work out all fiber configurations possibly admitting more than 20 lines as fiber components:
$$
\begin{array}{c|c}
\#\{\text{lines in fibers of } \psi\} & \text{fiber configuration}\\
\hline
24 & 6I_4\\
22 & 5I_4 + 2I_2\\
& 5I_4+I_3+I_1\\
& 5I_4+IV\\
21 & 5I_4+I_2+2I_1\\
& 5I_4+III\\
&  4I_4+2I_3+I_2
\end{array}
$$
To rule out all these configurations, 
we will employ structural Weierstrass form arguments specific to characteristic $2$
(which apply since we always switch to the Jacobian of $\psi$).
%we can use since otherwise $S$ were to contain at most 21 lines by Corollary \ref{cor:21}).
%to rule out the third configuration from Section \ref{s:1st} and thus complete the proof of Theorem \ref{thm:v17}.
Similar arguments have been applied to the particular problem of maximal singular fibers of elliptic K3 surfaces
in \cite{SS}.
%
%We emphasize that the height method from the previous section does not work for this configuration
%since it actually occurs (starting from a line of the second kind) by Lemma \ref{lem:Z}.

\subsection{General set-up}

In characteristic $2$,
an elliptic curve with given non-zero j-invariant $j$ can be defined 
by a Weierstrass form over a given field $K$ by
\begin{eqnarray}
\label{eq:WF}
y^2 + xy = x^3 + \frac 1j.
\end{eqnarray}
As usual, this is unique up to quadratic twist, 
but here twists occur in terms of an extra summand $Dx^2$,
with $K$-isomorphic surfaces connected by the Artin-Schreier map $z\mapsto z^2+z$ over $K$:
\begin{eqnarray}
\label{eq:D}
y^2 + xy = x^3 + Dx^2 + \frac 1j.
\end{eqnarray}
The main approach now consists in substituting a conjectural j-invariant,
given as quotient
\begin{eqnarray}
\label{eq:j}
j = a_1^{12}/\Delta
\end{eqnarray}
associated to the usual integral Weierstrass form
\[
y^2 + a_1xy + a_3 y = x^3 + a_2x^2 + a_4x+a_6.
\]
Converting the twisted form of \eqref{eq:WF} with j-invariant from \eqref{eq:j}
to an integral model $E$,
we arrive at the Weierstrass form
\begin{eqnarray}
\label{eq:E}
y^2 + a_1^2xy = x^3 +D'x^2 + \Delta
\end{eqnarray}
which outside very special cases will be non-minimal at the zeroes of $a_1$.
Then minimalizing is achieved by running Tate's algorithm \cite{Tate}
which consequently gives relations between the coefficients of $a_1, D'$ and $\Delta$,
or in some cases like ours leads to a contradiction.
By inspection of \eqref{eq:E}, the polynomial $a_1$ encodes singular or supersingular fibers.
For immediate use, we record the 
following criterion which is borrowed from \cite[Lem. 2.4 (a)]{Schweizer}:

\begin{Lemma}
\label{lem:ss}
Assume that there is a supersingular place which is not singular.
Locating it at $t=0$, the $t$-coefficient of $\Delta$ has to vanish.
\end{Lemma}

\begin{proof}
By assumption, $a_1=ta_1'$, so the integral Weierstrass form \eqref{eq:E} reads
\[
y^2 + t^2a_1'^2xy = x^3 +D'x^2 + \Delta.
\]
Writing $\Delta=d_0+d_1t+\hdots$,
the fiber of the affine Weierstrass form at $t=0$ has a singular point at $(0,\sqrt{d_0})$.
Since $t=0$ is a place of good reduction, the Weierstrass form is non-minimal.
From Tate's algorithm \cite{Tate},
this translates as $(0,\sqrt{d_0})$ being in fact a surface singularity.
Equivalently, $d_1=0$ by Jacobi's criterion.
\end{proof}

\begin{Example}
\label{ex:21}
There cannot be a rational elliptic surface in characteristic $2$
with singular fibers $5 I_2, 2 I_1$.
Otherwise, $\Delta=\Delta_5^2(t^2+at+b)$ for some squarefree degree 5 polynomial $\Delta_5$ and with $a\neq 0$.
Since the surface is semi-stable, Lemma \ref{lem:ss} kicks in to show that $a=0$, contradiction.
\end{Example}

\begin{Remark}
Note that the criterion of Lemma \ref{lem:ss}
applies after any M\"obius transformation fixing $0$,
and to any supersingular place that is not singular.
Tracing the non-minimality argument further through Tate's algorithm,
one can, for instance, show that there do not exist elliptic fibrations in characteristic $2$
with configuration of singular fibers $4I_3 + 6I_2$
(as occuring on Schur's quartic over $\C$).
\end{Remark}

\subsection{Non-existence of the configurations $6I_4, 5I_4+2I_2, 5I_4+IV, 5I_4+III$}
\label{ss:72half}

In each of the said cases,
the j-invariant \eqref{eq:j} is a perfect square.
Equivalently, $\psi$ arises from another elliptic fibration by a purely inseparable base change.
(To see this, apply the variable transformation $y\mapsto y+Dx$ to \eqref{eq:D}.)
In the first two cases, this would lead to a rational elliptic surface with five $I_2$ fibers;
this cannot exist by \cite{Lang} (which can be checked independently as in Example \ref{ex:21} or \ref{ss:73}).
Similarly, the configuration $5I_4+IV$ cannot arise at all
because fibers of type $IV$ are only related to type $IV^*$ by inseparable base change,
so that the Euler-Poincar\'e characteristics would not sum up to a multiple of $12$.

For the last configuration, since the fiber of type $III$ comes with wild ramification of index $1$
(by the Euler-Poincar\'e characteristic formula), it can only arise from a singular fiber
of the same ramification index and total contribution to the discriminant congruent to $2$ mod $6$.
By \cite[Prop.~5.1]{SS}, this uniquely leads to Kodaira type $I_1^*$,
but then again with the five $I_2$ fibers the Euler-Poincar\'e characteristics do not sum up to a multiple of $12$.

\subsection{Non-existence of the configurations $5I_4+I_3+I_1, 5I_4+I_2+2I_1, 4I_4+2I_3+I_2$}
\label{ss:73}

Each of these configurations is semi-stable, so Lemma \ref{lem:ss} applies
with supersingular (smooth) place at $t=0$.
In the first case, for instance, we can locate the $I_3$ fiber at $\infty$
and write affinely
\[
\Delta = \Delta_5^4(t-\lambda)
\]
for some squarefree polynomial $\Delta$ of degree $5$.
But then $t\mid \Delta_5$ by Lemma \ref{lem:ss}, contradiction.
The other two configurations can be ruled out completely analogously.

\noindent
{\sl Proof of Proposition~\ref{prop:60}:}
Altogether it follows from $\S$~\ref{ss:72half} and $\S$~\ref{ss:73} that the fibers of $\psi$ may contain at most 20 lines.
In view of the upper bound
\[
\#\{\text{lines on } S\} \leq \#\{\text{lines in fibers of } \psi\} + 40,
\]
this completes the proof of Proposition~\ref{prop:60}.
\qed

\section{Proof of Theorem \ref{thm} in the squarefree case}
\label{s:sqfree}

Throughout this section, we work with a smooth quartic $S$ satisfying the following assumption:

\begin{Assumption}
\label{ass7}
$S$ contains neither a triangle nor a star composed of 3 lines nor a square comprising 4 lines.
\end{Assumption}

In short, we will also call $S$ squarefree (meaning also trianglefree).
%The considerations of this section are partially inspired by Degtyarev/Veniani {\bf here citations}.
%In this section we assume that 
%\begin{equation} \label{eq-ass-no-coplanar}
%\mbox{ no hyperplane section of the quartic $S$ consists of four lines.} 
%\end{equation}
%Moreover, we say that the quartic $S$ is {\bf square-free} iff it contains no four lines that form the I$_4$ configuration.
%
%
%\begin{lemm} \label{lemma-pairfree}
%If $S$ contains no pairs of coplanar lines, then it contains at most $21$ lines.
%\end{lemm}
%\begin{proof}
%The claim follows immediately from the fact that the Picard number of $S$ cannot exceed $22$.
%\end{proof}
In the sequel we will use the following stronger version of \cite[Thm~1.1]{RS14}.
\begin{Lemma} \label{lemma-conic}
Let  $Q \subset S$ be an  irreducible conic which is coplanar with two lines on $S$.
Then $Q$ is met by at most $46$ lines on $S$. 
\end{Lemma}

\begin{proof}
The proof of the claim forms the last section of the proof of \cite[Prop.~1.3]{RS14}.  
\end{proof}

\begin{Proposition} \label{lemma-squarfree}
If $S$ is squarefree, then $S$ contains at most $60$ lines.
\end{Proposition}

\begin{proof}
By Corollary \ref{cor:21} we can assume that $S$ contains a pair of lines $\ell_1 \neq \ell_2$ that intersect.
Let $\ell_2, \ell_3, \ell_4, \ldots, \ell_{k_1+1}$ be  the lines on $S$ that meet $\ell_1$, and let $\ell_{i,1}, \ldots, \ell_{i,k_{i}}$ be the lines 
that intersect $\ell_i$ for $i \geq 2$.
After reordering the lines, we assume $\ell_1= \ell_{i,1}$. % and put $a^{i_1,i_2}_{j_1,j_2} :=  \ell_{i_1,j_1}.\ell_{i_2,j_2}$. 

Suppose that $(k_1+k_i) \leq 16$ for an $i \in \{2, \ldots, k_1+1\}$. 
Then, by Lemma~\ref{lemma-conic} the irreducible conic in $|{\mathcal O}_{S}(1)-\ell_1-\ell_i|$ is met by at most $44$ lines on $S$
other than $\ell_1, \ell_i$,
and the claim of the proposition follows immediately, because every line on $S$ meets  either $\ell_1$ or $\ell_i$ or the conic. 

Otherwise, we assume to the contrary that $S$ contains more than 60 lines.
After exchanging $\ell_1, \ell_i$, if necessary,
and iterating the above process, we may assume using Lemma \ref{lem:v12} that
\begin{eqnarray}
\label{eq:k's}
5 \leq k_1 \leq k_i   \; \mbox{ and } \; (k_1+k_i) \geq 17  \;  \mbox{ for } i \geq 2 \, . 
\end{eqnarray}
In particular, we always have the lines $\ell_2,\hdots,\ell_6$ on $S$,
and $k_i\geq 9$ for $i\geq 2$.
Assumption \ref{ass7} guarantees that
\begin{eqnarray} \label{eq-firstcondition}
&& \ell_{j_1}.\ell_{j_2} = 0 \mbox{ whenever } 2 \leq j_1 < j_2 \nonumber \\
&& \ell_{i,j_1}.\ell_{i,j_2} = 0 \mbox{ whenever } 1 \leq j_1 < j_2 \mbox{ and } i \geq 2 \nonumber \\
&& \ell_{i,j} \text{ meets at most one of the } \ell_{m,n} \;\;\forall i,j,m\geq 2. 
\end{eqnarray}
Consider the divisor $D=2\ell_2+\ell_{2,2}+\hdots+\ell_{2,5}$ of Kodaira type $I_0^*$.
Then $|D|$ induces an genus one fibration fibration 
\[
\psi: S \to \PP^1
\]
with $D$ as fiber.
Naturally, $\ell_3,\hdots,\ell_{k_1+1}$, being perpendicular to $D$, also appear as fiber components,
but we can say a little more.
Namely, by \eqref{eq-firstcondition}, each $\ell_j$ (for $2<j\leq k_1+1$) comes with at least $(k_j-5)$ adjacent lines,
say $\ell_{j,6},\hdots,\ell_{j,k_j}$ which are also perpendicular to $D$.
Since $k_j\geq 9$, the divisor $2\ell_j+\ell_{j,6}+\hdots+\ell_{j,9}$ gives another $I_0^*$ fiber of $\psi$;
in particular, there is no space for any further fiber components, i.e.~$k_j=9$ for all $j=3,\hdots,k_1+1$.
Therefore we obtain $k_1\geq 8$ from \eqref{eq:k's},
so $\psi$ admits at least $8$  fibers of type $I_0^*$.
The sum of their contributions to the Euler-Poincar\'e characteristic of $S$ clearly exceeds $e(S)=24$,
even if $\psi$ were to be quasi-elliptic.
This gives the desired contradiction and thus concludes the proof of Proposition~\ref{lemma-squarfree}.
\end{proof}

\begin{Conclusion} 
We treated the triangle case (see Section~\ref{s:triangle}), the square case (Prop.~\ref{prop:60}) and the squarefree case 
(Prop.~\ref{lemma-squarfree}),
so the proof of Theorem~\ref{thm} is complete. \qed
\end{Conclusion}

\section{Example with 60 lines}
\label{s:ex}

This section gives an explicit example 
of a smooth quartic surface defined over the finite field $\F_4$
which contains 60 lines over $\F_{16}$.
Motivated by \cite{RS-Barth}
where a pencil of quintic surfaces going back to Barth was studied,
we consider geometrically  irreducible quartic surfaces with an
 action by the symmetric group $S_5$ of order $120$.
These come in a one-dimensional pencil
which can be expressed in elementary symmetric polynomials $s_i$ of degree $i$
in the homogeneous coordinates  of $\PP^4$ as
\[
S_\mu = \{s_1 = s_4 + \mu s_2^2 = 0 \}\subset\PP^4.
\]
There are 60 lines at the primitive third roots of unity as follows.
Let $\alpha\in\F_{16}$ be a fifth root of unity,
i.e. $\alpha^4+\alpha^3+\alpha^2+\alpha+1=0$.
Then $\mu_0=1+\alpha^2+\alpha^3$ is a cube root of unity,
and $S_{\mu_0}$ contains the line $\ell$ given by
\[
\ell=
\{s_1=x_3+x_2+(\alpha^3+\alpha+1)x_1=x_4+(\alpha^3+\alpha^2+\alpha+1)x_2+\alpha x_1=0\}.
%{t = la, x[3] = ca*x[1]+x[2], x[4] = a*x[1]+ba*x[2]}
\]
The $S_5$-orbit of $\ell$ consists of exactly 60 lines
which span $\NS(S_{\mu_0})$ of rank $20$ and discriminant $-55$.
We are not aware of any other smooth quartic surface in $\PP^2_{\bar\F_2}$ with 60 or more lines.

In fact, the surface $S_{\mu_0}$ can be shown to lift to characteristic zero 
together with all its 60 lines at $\hat\mu=-\frac 3{10} \pm\frac{\sqrt{-11}}{10}$
(although proving this is not as easy as it might seem).
The lines are then defined over the Hilbert class field $H(-55)$
as follows. Let $\alpha\in \bar\Q$ satisfy
\[
\alpha^4+18\alpha^2+125=0 \;\;\; \text{ so that } \;\;\; \hat\mu=\dfrac{\alpha^2+3}{20}.
\]
The quartic at $\hat\mu$ contains the line
\[
\ell=
\{s_1=x_3-ax_1-bx_2 = x_4-cx_1+x_2=0\}
\]
where
\[
2c+\alpha+1 = 
40b+\alpha^3-5\alpha^2+3\alpha-55=
  40a-\alpha^3-5\alpha^2-23\alpha-35=0.
\]
As before,  its $S_5$-orbit comprises all 60 lines on $S_{\hat\mu}$.

\subsection*{Acknowledgements}

We thank Davide Veniani for helpful comments and discussions,
and the anonymous referee for  valuable suggestions.
This paper was completed at the workshop on K3 surfaces and related topics at KIAS in November 2015.
We thank the organizers JongHae Keum and Shigeyuki Kond\=o for the great hospitality.

%
%We are indebted to Wolf Barth for sharing his insights on the subject
%starting more than 10 years ago.
%Thanks to Achill Sch\"urmann for helpful discussions on quadratic forms,
%and to Tomasz Szemberg and Davide Veniani for pointing out an error in an earlier version
%related to Proposition \ref{prop:=64}.
%We are grateful to Igor Dolgachev, Duco van Straten and the anonymous referee 
%for their valuable comments.
%This project was started in March 2011 when MS enjoyed the hospitality of the Jagiellonian University in Krakow.
%Special thanks to S\l awomir Cynk.
%
%

\end{document}